\documentclass{llncs}
\usepackage{epsfig}

\newcommand{\R}{{I\!\!R}}

\newcommand{\be}{\begin{equation}}
\newcommand{\ee}{\end{equation}}
\newcommand{\ba}{\begin{eqnarray}}
\newcommand{\ea}{\end{eqnarray}}
\newcommand{\nn}{\nonumber}
\newcommand{\la}{\label} 
\newcommand{\ep}{\Delta t}

\newcommand{\e}{{\rm e}} 
\newcommand{\ct}{{\cal T}}

\newcommand{\bq}{{\bf q}}

\newcommand{\bac}{{\bf a}}

\def\R{{\rm I}\! {\rm R}}
\newcommand{\bp}{{\bf p}}
\newcommand{\pa}{\partial}

\newcommand{\bv}{{\bf v}}

\def\R{{\rm I}\! {\rm R}}

\def\X{{\bf X}}

\usepackage{amsmath,amssymb}
\usepackage{amsfonts}
\usepackage{makeidx}

\def\O{{\mathcal O}}

\begin{document}

\pagestyle{headings}

\title{Multi-product expansion for Nonlinear Differential Equations}
\author{J\"urgen Geiser}
\institute{
\email{juergen.geiser@uni-greifswald.de}}
\maketitle

\begin{abstract}

In this paper we discuss the extention of MPE methods to 
nonlinear differential equations.
We concentrate on nonlinear systems of differential equations and
generalize the recent MPE method, see \cite{chin2010}.

\end{abstract}



\section{Introduction}

In this paper we concentrate on approximation
to the solution of the nonlinear evolution equation, e.g. 
nonlinear evolution equation,
\begin{eqnarray}
\label{equ_exact_1}
&&  \partial_t \; u = F(u(t)) = A(u(t)) + B(u(t)),  \; u(0) = u_0 ,
\end{eqnarray}
with the unbounded operators $A: D(A) \subset \X \rightarrow \X$
and  $B: D(B) \subset \X \rightarrow \X$.
We have further $F(v) = A(v) + B(v), v \in D(A) \cap D(B) $.

We assume to have suitable chosen subspaces of the underlying
Banach space $(X, || \cdot ||_{X})$ such that $ D(F) = D(A) \cap D(B) \neq \emptyset  $

For such equations, we concentrate on
comparing the higher order methods to Suzuki's schemes.
Here the Suzuki's methods apply factorized symplectic algorithms
with forward derivatives, see \cite{chin02}, \cite{chin06}.

The exact solution of the evolution problem (\ref{equ_exact_1}) is given as:
\begin{eqnarray}
\label{equ1}
&&  u(t) = {\cal E}_F(t, u(0)) , \; 0 \le t \le T ,
\end{eqnarray}
with the evolution operator ${\cal E}_F$ depending on the actual time $t$
and the initial value $u(0)$.

We apply a formal notation given as:
\begin{eqnarray}
\label{equ1}
&&  u(t) = \exp( t D_F)  u(0) , \; 0 \le t \le T ,
\end{eqnarray}

Here the evolution operator $\exp(t D_F)$ and the Lie-derivative $D_F$ associated with $F$ are given as:
\begin{eqnarray}
\label{equ1}
&&  \exp( t D_F) G v = G({\cal E}_F(t, v)) , \; 0 \le t \le T , \; D_F G v = G'(V) F(v)
\end{eqnarray}
for any unbounded nonlinear operator $G: D(G) \subset X \rightarrow X$ with
Frechet derivative $G'$.

\section{Nonlinear Exponential operator splitting methods}

In the course of devising numerical algorithms for solving
 the prototype nonlinear differential equations
\be
\partial_t u = A(u) + B(u), \qquad u(0) = u_0 ,
\la{equ1}
\ee
where $A$ and $B$ are non-commuting operators,

Strang\cite{strang68} proposed two second-order algorithms corresponding 
to approximating
\be
{\cal T}(\ep)=\e^{\ep(D_A + D_B)}
\la{expab}
\ee
either as
\be
S(\ep)=\frac12\left( \e^{\ep D_A}\e^{\ep D_B}+\e^{\ep D_B}\e^{\ep D_A}\right)
\la{str1}
\ee
or as
\be
S_{AB}(\ep)=\e^{(\ep/2) D_B}\e^{\ep D_A} \e^{(\ep/2) D_B}.
\la{str2}
\ee
Following up on Strang's work, Burstein and Mirin\cite{bur70} suggested
that Strang's approximations can be generalized to higher orders
in the form of a multi-product expansion (MPE),
\be
{\rm e}^{\ep(D_A + D_B)}=\sum_k c_k
\prod_{i}
{\rm e}^{a_{ki}\ep D_A}{\rm e}^{b_{ki}\ep D_B}
\label{mprod} 
\ee
and gave two third-order approximations
\be
D(\ep)=\frac43 \left(\frac{S_{AB}(\ep)+S_{BA}(\ep)}{2}\right)-\frac13 S(\ep)
\la{dun}
\ee
where $S_{BA}$ is just $S_{AB}$ with $A\leftrightarrow B$ interchanged, and
\be
B_{AB}(\ep)=\frac98 \e^{(\ep/3) D_A}\e^{(2\ep/3) D_B}\e^{(2\ep/3) D_A}\e^{(\ep/3) D_B}
-\frac18 \e^{\ep D_A}\e^{\ep D_B}.
\la{bmsch}
\ee 
They credited J. Dunn for finding the decomposition $D(\ep)$ and
noted that the weights ${c_k}$ are no longer positive beyond 
second order. Thus the  stability of the entire algorithm can no longer be 
inferred from the stability of each component product.

\section{Nonlinear Multi-product decomposition}
\label{mpedec}

The multi-product decomposition (\ref{mprod}) is obviously more complicated 
than the single product splitting. 

By the way after Burstein and Mirin, Sheng\cite{sheng89} proved their
observation that beyond second-order, $a_{ki}$, $b_{ki}$ and $c_k$ cannot all be
positive.
Some possibilities are the idea of complex coefficients
to obtain a higher order scheme.

For general applications, including solving time-irreversible problems, 
one must have $a_{ki}$ and $b_{ki}$ positive.

Here we show, that we obtain a extrapolation scheme, that can 
overcome such problems.

Therefore every single product in (\ref{mprod}) can at most be 
second-order\cite{sheng89,suzuki91}. But such a product is
easy to construct, because every left-right symmetric single product {\it is} 
second-order. Let ${\cal T}_S(h)$ be such a product with
$\sum_i a_{ki}=1$ and $\sum_i b_{ki}=1$, then
${\cal T}_S(h)$ is time-symmetric by construction, 
\be
{\cal T}_S(-h){\cal T}_S(h)=1,
\ee
implying that it has only odd powers of $h$
\be
{\cal T}_S(\ep)=\exp(\ep (D_A + D_B)+\ep^3 E_3+\ep^5 E_5+\cdots)
\la{secerr}
\ee
and therefore correct to second-order.
(The error terms $E_i$ are nested commutators of $D_A$ and $D_B$
depending on the specific form of $\ct_S$.)
This immediately suggests that the $k$th power of ${\cal T}_S$ at 
step size $h/k$ must have the form 
\be
{\cal T}_S^k(\ep/k) 
=\exp(\ep (D_A+D_B)+k^{-2}\ep^3 E_3+k^{-4}\ep^5 E_5+\cdots),
\la{thk}
\ee
and can serve as a basis for the multi-production expansion (\ref{mprod}).
The simplest such symmetric product is  
\be
{\cal T}_2(h)=S_{AB}(h) \quad{\rm or}\quad {\cal T}_2(h)=S_{BA}(h).
\ee
If one naively assumes that
\be
{\cal T}_2(h)=\e^{\ep(D_A+D_B)}+ Ch^3+Dh^4+\cdots,
\ee
then a Richardson extrapolation would only give
\be
\frac1{k^2-1}\left[k^2 \ct_2^k(h/k)-\ct_2(h)\right]=\e^{\ep(D_A+D_B)}+ O(h^4),
\ee
a third-order\cite{schat94} algorithm. However, because the error structure
of $\ct_2(h/k)$ is actually given by (\ref{thk}), one has
\be
{\cal T}_2^k(h/k)=\e^{\ep(D_A + D_B)}+ k^{-2}h^3E_3+\frac12 k^{-2}h^4[(D_A+D_B)E_3+E_3(D_A+D_B)]+O(h^5),
\ee
and {\it both} the third and fourth order errors can be eliminated simultaneously, 
yielding a fourth-order algorithm. Similarly, the leading $2n+1$ and $2n+2$ order 
errors are multiplied by $k^{-2n}$ and can be eliminated at the same time. Thus 
for a given set of $n$ whole numbers $\{k_i\}$ one can have a $2n$th-order approximation  
\be
{\rm e}^{\ep(D_A+D_B)}
=\sum_{i=1}^n c_i{\cal T}_2^{k_i}\left(\frac\ep{k_i}\right)
+O(h^{2n+1}).
\la{mulexp}
\ee
provided that $c_i$ satisfy the simple Vandermonde equation: 
\be
\left(
\begin{array}{c c c c c}
1 & 1 & 1 & \ldots & 1 \\
k_1^{-2} &  k_2^{-2} & k_3^{-2} & \ldots & k_n^{-2} \\
k_1^{-4} &  k_2^{-4} & k_3^{-4} & \ldots & k_n^{-4} \\
\ldots & \ldots & \ldots & \ldots & \ldots \\
k_1^{-2 (n - 1)} &  k_2^{-2 (n - 1)} & k_3^{-2 (n - 1)} & \ldots & k_n^{-2 (n - 1)}
\end{array}
\right)
\left(
\begin{array}{c}
c_1 \\
c_2 \\
c_3 \\
\ldots \\
c_n
\end{array}
\right)
=
\left(
\begin{array}{c}
1 \\
0 \\
0 \\
\ldots \\
0
\end{array}
\right)
\la{vand}
\ee
Surprising, this equation has closed form solutions\cite{chin08} for all $n$
\be
c_i=\prod_{j=1 (\ne i)}^n\frac{k_i^2}{k_i^2-k_j^2}.
\la{coef}
\ee
The natural sequence $\{k_i\}=\{1, 2, 3 \,...\, n\}$ produces a $2n$th-order
algorithm with the minimum $n(n+1)/2$ evaluations of $\ct_2(h)$.
For orders four to ten, one has explicitly:
\be
{\cal T}_4(\ep)=-\frac13{\cal T}_2(\ep)
+\frac43{\cal T}_2^2\left(\frac\ep{2}\right)
\la{four}
\ee
\be
{\cal T}_6(\ep)=\frac1{24} {\cal T}_2(\ep)
-\frac{16}{15}{\cal T}_2^2\left(\frac\ep{2}\right)
+\frac{81}{40}{\cal T}_2^3\left(\frac\ep{3}\right)
\la{six}
\ee
\be
{\cal T}_8(\ep)=-\frac1{360} {\cal T}_2(\ep)
+\frac{16}{45}{\cal T}_2^2\left(\frac\ep{2}\right)
-\frac{729}{280}{\cal T}_2^3\left(\frac\ep{3}\right)
+\frac{1024}{315}{\cal T}_2^4\left(\frac\ep{4}\right)
\la{eight}
\ee
\ba
&&{\cal T}_{10}(\ep)=\frac1{8640} {\cal T}_2(\ep)
-\frac{64}{945}{\cal T}_2^2\left(\frac\ep{2}\right)
+\frac{6561}{4480}{\cal T}_2^3\left(\frac\ep{3}\right)\nn\\
&&\qquad\qquad\quad-\frac{16384}{2835}{\cal T}_2^4\left(\frac\ep{4}\right)
+\frac{390625}{72576}{\cal T}_2^5\left(\frac\ep{5}\right).
\la{ten}.
\ea
As shown in Ref.\cite{chin08}, $\ct_4(h)$ reproduces Nystr\"om's fourth-order 
algorithm with three force-evaluations
and $\ct_6(h)$ yielded a new
sixth-order Nystr\"om type algorithm with five force-evaluations.

\subsection{Error Analysis of the Nonlinear Multiproduct Expansion}

The error analysis is based on the following ideas:
\begin{itemize}
\item Definition of a new operator (linear operator)
\item Definition of the derivatives
\item Derivation of the splitting errors
\end{itemize}

The nonlinear equation is given as:
\begin{eqnarray}
\label{equ1}
&&  \partial_t \; u = F(u(t)) = A(u(t)) + B(u(t)),  \; u(0) = u_0 ,
\end{eqnarray}

We associate a new operator $\hat{F}$ which is a linear Lie operator:

\begin{definition}
For the given operator $F$ we associate a new operator,
denoted $\hat{F}$.

This Lie-operator acts on the space of the differentiable operators of the
type:

$ S \rightarrow S$ ,\\

and maps each operator $G$ into the new operator $\hat{F}(G)$, such that
for any element $c \in S$:
\begin{eqnarray}
\label{equ1}
&&  (\hat{F}(G))(c) = ( G'(c) \circ  F)(c) ,
\end{eqnarray}
\end{definition}

Here the derivatives are given as:

\begin{definition}
The $k$-th power of the Lie-operator $\hat{F}$ applied to some
operator $G$ can be expressed as the $k$-th derivative of $G$, that is the
relation
\begin{eqnarray}
\label{equ1}
&&  (\hat{F}^k(G))(c) = \frac{\partial^k G(c)}{\partial t^k}  ,
\end{eqnarray}
is valid for all $k = 1, 2, \ldots$.
\end{definition}

Let us splitt the operator $F$ into the sum $F_1 + F_2$

The splitting error for the A-B splitting is given as:

\begin{eqnarray}
\label{equ1}
&&  Err_{A-B} = \exp(\tau \hat{F}_1 + \hat{F}_2)(I) -  (\exp(\tau \hat{F}_1) \exp(\tau \hat{F}_2))(I)(c) ,
\end{eqnarray}
if we set in the commutator we obtain for $F_1$ and $F_2$
\begin{eqnarray}
\label{equ1}
&&  Err_{A-B, 1,2} = \tau (F_2'(c)  \circ F_1)(c) -F_1'(c)  \circ F_2)(c)) + O(\tau^2),
\end{eqnarray}

The Strang Splitting is given as:
\begin{eqnarray}
\label{equ1}
Err_{Strang} && = \exp(\tau \hat{F}_1 + \hat{F}_2)(I) \\
                && -  (\exp(\tau/2 \hat{F}_1) \exp(\tau \hat{F}_2) \exp(\tau/2 \hat{F}_1))(I)(c) , \nonumber
\end{eqnarray}
if we set in the commutator we obtain for $F_1$ and $F_2$
\begin{eqnarray}
\label{gleich_kap3_11}
&&  Err_{Strang, 1,2} = \frac{1}{24} \tau^2 ([F_2, [F_2, F_1]](c) - 2 [F_1, [F_1, F_2]](c)) + O(\tau^4) \; . 
\end{eqnarray}
with the commutator: $[F_1, F_2](c) = (F_2'(c)  \circ F_1)(c) - (F_1'(c)  \circ F_2)(c) $.

We derive the nonlinear MPE based on the definition of the linearized Lie-operator.

The derivation of the MPE method is given as:
\be
{\cal T}_2^k(h/k)=\e^{\ep(D_A + D_B)}+ k^{-2}h^3E_3+\frac12 k^{-2}h^4[(D_A+D_B)E_3+E_3(D_A+D_B)]+O(h^5),
\ee
and {\it both} the third and fourth order errors can be eliminated simultaneously, while $E_3$ and $E_4$ are given as commutators of the derivatives of $A$ and $B$, see \cite{chin2010}.

\section{Nonlinear Operator splitting methods}
\label{oper} 

In the literature, there are various types of splitting
methods.
 We mainly consider the following operators splitting schemes in this study:\\

{1. Sequential operator splitting: A-B splitting}
\begin{eqnarray}
\label{seqsp1}
 && \frac{\partial c^{*}(t)}{\partial t} = A(c^{*}(t)) \;\;\;\;\;
\mbox{with} \; t \in [t^{n},t^{n+1}]\;\;\;\ \mbox{and}\;\;\;\; c^{*}(t^n) = c_{sp}^{n} \\
\label{seqsp2}
 && \frac{\partial c^{**}(t)}{\partial t} = B( c^{**}(t)) \;\;\;\; \mbox{with} \; t \in [t^{n},t^{n+1}] \;\;\;\;
\mbox{and} \;\;\; c^{**}(t^n) =c^{*}(t^{n+1}),
\end{eqnarray}
for $n=0,1,...,N-1$ whereby $c_{sp}^{n}=c_0$ is given from
(\ref{Cauchy}). The approximated split solution at the point
$t=t^{n+1}$ is defined as $c_{sp}^{n+1}=c^{**}(t^{n+1}).$

{2. Strang-Marchuk operator splitting : A-B-A splitting }
\begin{eqnarray}
\label{str1}
 && \frac{\partial c^{*}(t)}{\partial t} = A(c^{*}(t)) \;
\mbox{with} \; t \in [t^{n},t^{n+1/2}] \;\;\;\; \mbox{and}\;\;\;\; c^{*}(t^n) = c_{sp}^{n} \\
 \label{str2} && \frac{\partial c^{**}(t)}{\partial t} =
B(c^{**}(t)) \; \mbox{with} \; t \in [t^{n},t^{n+1/2}] \;\;\;\;
\mbox{and}\;\;\;\; c^{**}(t^n) =
c^{*}(t^{n+1/2}) \; , \\
\label{str3}&& \frac{\partial c^{***}(t)}{\partial t} = A(c^{*}(t))
\; \mbox{with} \; t \in [t^{n+1/2},t^{n+1}] \;\;\;\; \mbox{and}\;
c^{***}(t^{n+1/2}) = c^{**}(t^{n+1}),
\end{eqnarray}
where $t^{n+1/2}= t^{n} + \tau/2,$   $\tau$ is the local time step.
The approximated split solution at the point $t=t^{n+1}$
is defined as $c_{sp}^{n+1}=c^{***}(t^{n+1}).$\\

{3. Iterative splitting with respect to one operator}\\

\begin{eqnarray}
&& \frac{\partial c_i(t)}{\partial t} = A(c_i(t)) \; + \; B(c_{i-1}(t)), \; \mbox{with} \; \; c_i(t^n) = c^{n}, i = 1, 2, \ldots,
m \; \label{iterone}
\end{eqnarray}

{4. Iterative splitting with respect to alternating operators}\\

\begin{eqnarray}
 && \frac{\partial c_i(t)}{\partial t} = A(c_i(t)) \; + \; B(c_{i-1}(t)), \; \label{itertwo_1}
\mbox{with} \; \; c_i(t^n) = c^{n} , \\
&& i = 1, 2, \ldots, j \; , \nonumber\\
 && \frac{\partial c_{i}(t)}{\partial t} =
A(c_{i-1}(t)) \; + \; B(c_{i}(t)), \;\;\;\;\; \mbox{with} \; \;
c_{i+1}(t^n) =
c^{n} , \label{itertwo_2} \\
&& i = j+1, j+2, \ldots, m \; , \nonumber
\end{eqnarray}

Here, $c_{0}(t^n) = c^n \; , \; c_{-1} = 0$ and $c^n$ is the known
split approximation at the time level $t=t^{n}$. The split
approximation at the time-level $t=t^{n+1}$ is defined as
$c^{n+1}=c_{2m+1}(t^{n+1})$.

\subsection{Error Analysis of the Classical Splitting Schemes}

We associate a new operator $\hat{F}$ which is a linear Lie operator:

\begin{definition}
For the given operator $F$ we associate a new operator,
denoted $\hat{F}$.

This Lie-operator acts on the space of the differentiable operators of the
type:

$ S \rightarrow S$ ,\\

and maps each operator $G$ into the new operator $\hat{F}(G)$, such that
for any element $c \in S$:
\begin{eqnarray}
\label{equ1}
&&  (\hat{F}(G))(c) = ( G'(c) \circ  F)(c) ,
\end{eqnarray}
\end{definition}

Here the derivatives are given as:

\begin{definition}
The $k$-th power of the Lie-operator $\hat{F}$ applied to some
operator $G$ can be expressed as the $k$-th derivative of $G$, that is the
relation
\begin{eqnarray}
\label{equ1}
&&  (\hat{F}^k(G))(c) = \frac{\partial^k G(c)}{\partial t^k}  ,
\end{eqnarray}
is valid for all $k = 1, 2, \ldots$.
\end{definition}

Let us splitt the operator $F$ into the sum $F_1 + F_2$

The splitting error for the A-B splitting is given as:

\begin{eqnarray}
\label{equ1}
&&  Err_{A-B} = \exp(\tau \hat{F}_1 + \hat{F}_2)(I) -  (\exp(\tau \hat{F}_1) \exp(\tau \hat{F}_2))(I)(c) ,
\end{eqnarray}
if we set in the commutator we obtain for $F_1$ and $F_2$
\begin{eqnarray}
\label{equ1}
&&  Err_{A-B, 1,2} = \tau (F_2'(c)  \circ F_1)(c) -F_1'(c)  \circ F_2)(c)) + O(\tau^2),
\end{eqnarray}

The Strang Splitting is given as:
\begin{eqnarray}
\label{equ1}
Err_{Strang} && = \exp(\tau \hat{F}_1 + \hat{F}_2)(I) \\
                && -  (\exp(\tau/2 \hat{F}_1) \exp(\tau \hat{F}_2) \exp(\tau/2 \hat{F}_1))(I)(c) , \nonumber
\end{eqnarray}
if we set in the commutator we obtain for $F_1$ and $F_2$
\begin{eqnarray}
\label{gleich_kap3_11}
&&  Err_{Strang, 1,2} = \frac{1}{24} \tau^2 ([F_2, [F_2, F_1]](c) - 2 [F_1, [F_1, F_2]](c)) + O(\tau^4) \; . 
\end{eqnarray}
with the commutator: $[F_1, F_2](c) = (F_2'(c)  \circ F_1)(c) - (F_1'(c)  \circ F_2)(c) $.

\section{Improving the Initialization of Operator Splitting Methods}\label{higher}

A delicate problem in splitting methods is to achieve sufficient accuracy
in the first splitting steps, see \cite{geiser2011}.

Next we discuss the extension to the improvement with
Zassenhaus formula.

\subsection{Higher order A-B splitting by Initialization }

The idea is based on the novel commutator that is given in Section \ref{oper}.

\begin{theorem}
We solve the initial value problem by using the method given in
equations (\ref{seqsp1}) and (\ref{seqsp2}). We assume bounded and
nonlinear operators $A$ and $B$.

The consistency error of the A-B splitting is $\mathcal{O}(t)$, then
we can improve the error of the A-B splitting scheme to
$\mathcal{O}(t^{p}), p> 1$ by improving the starting conditions
$c_0$ as

$$c_0= ( \pi_{j=2}^{p} \exp(\hat{C}_j t^j)) c_0 $$
where $\hat{C}_j$ is called a nonlinear Zassenhaus exponents 
that is given with the novel commutator, e.g. the linear
case is given in \cite{scholz06}.

The local splitting error of A-B splitting method
can be read as follows

\begin{eqnarray}
  \rho_n & =& (\exp(\tau_n (\hat{A}+\hat{B}))(I) - \exp(\tau_n \hat{B})(I) \exp(\tau_n \hat{A})(I) ) (c_{sp} ^{n}) \\ \nonumber
         & =& \hat{C}_T \,\, \tau_n^{p+1} + \mathcal{O}(\tau_n^{p+2}) \nonumber
\end{eqnarray}
where $\hat{C}_T$ is a function of nonlinear Lie brackets of $\hat{A}$ and $\hat{B}$.
\end{theorem}

\begin{proof}
Let us consider the subinterval [0, t], where $\tau=t$, the solution
of the subproblem (\ref{seqsp1}) is:
\begin{eqnarray}
  c^{*}( t) & =& \exp(t \hat{A})(I)(c_0)
\end{eqnarray}
after improving the initialization we have
\begin{eqnarray}
  c^{*}(t) & =& \exp(t \hat{A})(I) ( \pi_{j=2}^{p} \exp(\hat{C_j} t^j))(I)(c_0)
\end{eqnarray}
the solution of the subproblem (\ref{seqsp2}) becomes
 \begin{eqnarray}
  c^{**}(t) & =& \exp (t \hat{B})(I) \, \exp(t \hat {A})(I) ( \pi_{j=2}^{p} \exp(\hat{C_j} t^j))(I) c_0\\
                  & =& (\exp(\tau_n (\hat{B}+\hat{A}))(I)(c_0) + \mathcal{O}(t^{p+1})\nonumber
\end{eqnarray}
with the help of the Zassenhaus product formula.
\end{proof}

\begin{remark}
For example, the second order A-B splitting after improving the
initialization is
\begin{eqnarray}
  c^{**}(t) & =& \bigg(\exp (t \hat{B})(I) \, \exp(t \hat{A})(I) \, \exp ( - \frac{1}{2} t^2 [\hat{B}, \hat{A}])(I) \bigg)  (c_0) \nonumber \\
                  & =&  \exp (t (\hat{B} + \hat{A}))(I)(c_0) + \mathcal{O}(t^{3})
\end{eqnarray}
and the third order A-B splitting after improving the initialization
is
\begin{eqnarray}
  c^{**}(t) & =& \bigg( \exp (t \hat{B})(I)\, \exp(t \hat{A})(I) \, \exp ( - \frac{1}{2} t^2 [B, A])(I) \, \exp ( \frac{1}{6} t^3 [\hat{B},[\hat{B}, \hat{A}]](I)  \nonumber \\ 
&& - \frac{1}{3} [\hat{A},[\hat{A},\hat{B}]](I) ) \bigg) (c_0)  \\
                  & =& \exp (t (\hat{B} + \hat{A}))(I)(c_0) + \mathcal{O}(t^{4}),
\end{eqnarray}
where the commutator is given 
as $[A, B](c) = (B'(c)  \circ A)(c) - (A'(c)  \circ B)(c) $.

\end{remark}

\begin{remark}
The same idea can also be done with the Strang-Splitting method, see the linear
case in \cite{geiser2011}.
\end{remark}

\section{Alternative Approaches: Iterative Schemes for Linearization}

In the following, we discuss the fixed point iteration and Newton's method
as alternative approaches to linearize the nonlinear problems.

We solve the nonlinear problem:
\begin{eqnarray}
\label{nonlinear}
 F(x) = 0 ,
\end{eqnarray}
where $F: \R^n \rightarrow \R^n$.

\subsection{Fixed-point iteration}

The nonlinear equations can be formulated as fixed-point problems:
\begin{eqnarray}
\label{fix}
 x = K(x) ,
\end{eqnarray}
where $K$ is the fixed-point map and is nonlinear, e.g. $K(x) = x - F(x)$.

A solution of (\ref{fix}) is called fix-point of the map $K$.

The fix-point iteration is given as:
\begin{eqnarray}
\label{fix}
 x_{i+1} = K(x_i) ,
\end{eqnarray}
and is called {\it nonlinear Richardson iteration}, {\it Picard iteration},
or {\it the method of successive substitution}.

\begin{definition}
Let $\Omega \le \R^n$ and let $G: \Omega \rightarrow \R^m$.
$G$ is Lipschitz continuous on $\Omega$ with Lipschitz constant $\gamma$ if 
\begin{eqnarray}
\label{fix_2}
|| G(x) - G(y) || \le \gamma || x - y || ,
\end{eqnarray}
for all $x, y \in \Omega$.
\end{definition}

For the convergence we have to assume that $K$ be a contraction map on $\Omega$
with Lipschitz constant $\gamma < 1$.

\begin{example}

We apply the fix-point iterative scheme to decouple the
non-separable Hamiltonian problem.

\begin{eqnarray}
& &{\bf \dot q}_i = \frac{\partial H}{\partial \bp}(\bp_i,  \bq_{i-1}) , \\
& & {\bf \dot p}_i = - \frac{\partial H}{\partial \bq}(\bp_{i-1}, \bq_i) ,
\end{eqnarray}

while we have the initial condition for the fix-point iteration:

$ (\bp_0,  \bq_0) = (\bp(t^n),  \bq(t^n))$

we assume that we have convergent results after $i = 1 \ldots, m$ iterative steps or with the stopping criterion:

$ \max (|| \bp_{i+1} - \bp_{i} || , || \bq_{i+1} - \bq_i ||) \le err $,

while $|| \cdot ||$ is the Euclidean norm (or a simple vector-norm, e.g. $L_2$). 

\end{example}

\subsection{Newton's method}

We solve the nonlinear operator equation (\ref{nonlinear}).

While $F: D \subset X \rightarrow Y$ with the Banach spaces $X, Y$
is given with the norms $|| \cdot ||_X$ and $|| \cdot ||_Y$.
Let $F$ be at least once continuous differentiable,
further we assume $x_0$ is a starting solution of the
unknown solution $x^*$.  

Then the {\it successive linearization} lead to the
general Newton's method:
\begin{eqnarray}
\label{newton}
 F'(x_i) \Delta x_i = - F(x_i) ,
\end{eqnarray}
where $\Delta x_i = x_{i+1} - x_i$ and $i= 0,1,2, \ldots .$

The method derive the solution of a nonlinear problem by solving
{\it a sequence of linear problems of the same kind}.

\begin{remark}
The linearization methods can be applied to the MPE methods.
Non-separable Hamiltonian problems can decoupled to separable Hamiltonians,
see the ideas in \cite{chin2008_1}.
\end{remark}

\section{Numerical Examples}
\label{exp}

In the following section, we deal with experiments to
verify the benefit of our methods.
At the beginning, we propose introductory examples to
compare the methods. In the next examples, some ideas to 
applications in Burgers and Hamiltonian problems, are done.

\subsection{First Example: Burgers equation}

We deal with a 2D example where we can derive an
analytical solution.
\begin{eqnarray}
\label{test_1a}
&& \partial_{t} u  =  - u \partial_{x} u -  u \partial_y u + \mu ( \partial_{xx} u + \partial_{yy} u ) + f(x, y, t)
, \\
&& (x,y, t) \in \Omega \times [0, T]\nonumber\\
&& u(x, y, 0) = u_{\rm ana}(x,y,0),  \; (x,y) \in \Omega \\
&& \mbox{with} \; u(x,y,t) =  u_{\rm ana}(x,y,t) \; \mbox{on} \partial \Omega \times [0,T] , 
\end{eqnarray}
where $\Omega = [0,1] \times [0,1]$, $T = 1.25$, and $\mu$ is the viscosity.

The analytical solution 
is given as
\begin{eqnarray}
\label{test_2a}
&& u_{\rm ana}(x, y, t)  =  ( 1 + \exp(\frac{x + y - t}{2 \mu}))^{-1} ,
\end{eqnarray}
where $f(x,y,t) = 0$.\\

For the non-asymptotic case we compute the right-hand side as:
\begin{eqnarray}
\label{test_1b}
&& f(x, y, t) = - \partial_{t} u_{asym} - u_{asym} \partial_{x} u_{asym} -  u_{asym} \partial_y u_{asym} \\
&& + \mu ( \partial_{xx} u_{asym} + \partial_{yy} u_{asym} ) + f(x, y, t)\;
, \; \mbox{in} (x,y, t) \in [0,1]\times[0,1] \times [0, 1.25] \nonumber \\
&& u(x, y, 0) = u_{asym}(x,y,0),  (x,y) \in [0,1]\times[0,1] \\
&& \mbox{with} \; u(x,y,t) =  u_{asym}(x,y,t) \; \mbox{on} \partial \Omega \times [0,T] \; , 
\end{eqnarray}

We discretize with $\Delta x, \Delta y = 1/40$, $\Delta t = 0.001$.

\noindent The operators are given as: \\

$A(u)u =  - u \partial_{x} u -  u \partial_y u$, 
hence $A(u) = -u \partial_x - u \partial_y $ 
(the nonlinear operator), 

$B u = \mu ( \partial_{xx} u + \partial_{yy} u) + f(x, y, t)$ 
(the linear operator).\\

\noindent We apply the nonlinear Algorithm (\ref{itertwo_1})  to the first equation and obtain\\

$ A(u_{i-1}) u_{i} = - u_{i-i} \partial_x u_i - u_{i-i} \partial_y u_i$ and

$ B u_{i-1} =  \mu (\partial_{xx} + \partial_{yy} ) u_{i-1} + f$,\\

\noindent and we obtain linear operators, because $u_{i-1}$ is known from the
previous time step. \\

\noindent In the second equation we obtain by using Algorithm (\ref{itertwo_2}):\\

$ A(u_{i-1}) u_{i} = - u_{i-i} \partial_x u_i - u_{i-i} \partial_y u_i$ and

$ B u_{i+1} =  \mu (\partial_{xx} + \partial_{yy} ) u_{i+1} + f $,\\

\noindent and we have also linear operators.\\

\noindent The maximal error at end time $t = T$ is given as
\[{\rm err}_{\rm max} = |u_{\rm num} - u_{\rm ana}| = \max_{i= 1}^p |u_{\rm num}(x_i, t) - u_{\rm ana}(x_i, t) |,\]
the numerical convergence rate is given as 
\[ \rho = \log ({\rm err}_{h/2} / {\rm err}_{h} )/ \log(0.5) .\]
We have the following results, see Tables \ref{table_1} for different steps in time and space and different viscosities, see also the results in \cite{geiser_2010_1}.
\begin{table}[h!]
\begin{center}
\begin{tabular}{||c|c||c|c||c|c||}
\hline \hline
 $\Delta x=\Delta y$ & $\Delta t$& ${\rm err}_{L_1}$ & ${\rm err}_{\rm max}$ & $ \rho_{L_1} $& $ \rho_{\rm max} $  \\
 \hline 
1/10 & 1/10 &  0.0549  & 0.1867 &  & \\
1/20 & 1/10 &  0.0468  & 0.1599 & 0.2303 & 0.2234 \\
1/40 & 1/10 &  0.0418  & 0.1431 & 0.1630 & 0.1608 \\
 \hline 
1/10 & 1/20 &  0.0447  & 0.1626 &  & \\
1/20 & 1/20 &  0.0331  & 0.1215 & 0.4353 & 0.4210 \\
1/40 & 1/20 &  0.0262  & 0.0943 & 0.3352 & 0.3645 \\
 \hline 
1/10 & 1/40 &  0.0405  & 0.1551 &  & \\
1/20 & 1/40 &  0.0265  & 0.1040 & 0.6108 & 0.5768 \\
1/40 & 1/40 &  0.0181  & 0.0695 & 0.5517 & 0.5804 \\
\hline \hline
\end{tabular}\\
\caption{Numerical results for the Burgers equation with viscosity $\mu = 0.05$, initial condition $u_{0}(t) = c_n$, and two iterations per time step.}
\label{table_1}
\end{center}
\end{table}

\subsection{Separable Hamiltonian}

We deal with the evolution of any dynamical variable $u(\bq,\bp)$ (including $\bq$ and $\bp$ themselves)
is given by the
Poisson bracket,
\begin{equation}
\pa_tu(\bq,\bp)=
                 \Bigl(
		          {{\partial u}\over{\partial \bq}}\cdot
                  {{\partial H}\over{\partial \bp}}
				 -{{\partial u}\over{\partial \bp}}\cdot
                  {{\partial H}\over{\partial \bq}}
				                    \Bigr)=(A+B)u(\bq,\bp).
\label{peq}
\end{equation}

An example for a separable Hamiltonian is give as:
\begin{equation}
H(\bp,\bq)={\bp^2\over{2m}}+V(\bq),
\label{ham}
\end{equation}
$A$ and $B$ are Lie operators, or vector fields
\be
A=\bv\cdot\frac{\pa}{\pa\bq} \qquad B=\bac(\bq)\cdot\frac{\pa}{\pa\bv}
\la{shop} 
\ee
where we have abbreviated $\bv=\bp/m$ and $\bac(\bq)=-\nabla V(\bq)/m$.
The exponential operators $\e^{h A}$ and $\e^{h B}$ are then just shift operators. \\

 $S(h) = \e^{h/2 B} \e^{h A}  \e^{h/2 B}$  \\

That is also given as a Verlet-algorithm in the following scheme.

We start with $(\bq_0, \bv_0)^t = (\bq(t^{n}), \bv(t^{n}))^t $:

\begin{eqnarray}
(\bq_1, \bv_1)^t = \e^{h/2 B} (\bq_0, \bv_0)^t & = & ( I + \frac{1}{2} h \sum_i a(\bq) \frac{\partial}{\partial \bv_i}) (\bq_0, \bv_0)^t \\
& = & (\bq_0, \bv_0 + \frac{1}{2} h a(\bq_0) )^t ,
\end{eqnarray}

\begin{eqnarray}
(\bq_2, \bv_2)^t = \e^{h A} (\bq_1, \bv_1)^t & = & ( I + h \sum_i \bv_i \frac{\partial}{\partial \bq_i}) (\bq_1, \bv_1)^t \\
& = & (\bq_1 + h \bv_1 , \bv_1 )^t , 
\end{eqnarray}

\begin{eqnarray}
(\bq_3, \bv_3)^t = \e^{h/2 B} (\bq_2, \bv_2)^t & = & ( I + \frac{1}{2} h \sum_i a(\bq) \frac{\partial}{\partial \bv_i}) (\bq_2, \bv_2)^t \\
& = & (\bq_2, \bv_2 + \frac{1}{2} h a(\bq_1) )^t .
\end{eqnarray}

And the substitution is given the algorithm for one time-step $n \rightarrow n+1$:
\begin{eqnarray}
(\bq_3, \bv_3)^t = ( \bq_0 + h \bv_0 + \frac{h}{2} a(\bq_0) ,  \bv_0 + \frac{h}{2} a(\bq_0) + \frac{h}{2} a(\bq_0 + h \bv_0 + \frac{h}{2} a(\bq_0)) )^t ,
\end{eqnarray}
while $(\bq(t^{n+1}), \bv(t^{n+1}))^t = (\bq_3, \bv_3)^t$.

\begin{remark}
Here, we linearize with respect to the time-steps and assume to compute
a large time sequence.

The numerical error is given with $\O(h^2)$ based on the second order approaches
of the Strang-splitting.

\end{remark}

\section{Conclusions and Discussions }
\label{conc}

We have presented novel MPE approaches to nonlinear differential
equations. Based on the ideas of iterative splitting schemes, 
we could linearize
the numerical scheme and apply the linear MPE approach..
Numerical examples confirm the applications to nonlinear equations.
In the future we will focus us on the development of improved
MPE methods with respect to non-separable Hamiltonian problems.

\bibliographystyle{plain}

\end{document}